\documentclass[12pt,a4paper,oneside,openright]{amsart}

\usepackage{xcolor}
\usepackage[T1]{fontenc}
\usepackage[utf8]{inputenc}
\usepackage{amsmath,amsfonts,amssymb,amsthm}
\usepackage{tikz}
\usetikzlibrary{arrows,decorations.pathmorphing,positioning,backgrounds,fit,matrix}
\newtheorem{thm}{Theorem}[section]

\newtheorem{defn}[thm]{Definition}

\newtheorem{exmp}[thm]{Example}

\newtheorem{remark}[thm]{Remark}

\newtheorem{cor}[thm]{Corollary}

\begin{document}
\title{The $ L^1 $-Liouville property on graphs}
\author{Andrea Adriani and Alberto G. Setti}
\keywords{Infinite weighted graphs, $L^1$-Liouville property, stochastic completeness, curvature comparison}
\subjclass[2020]{05C63, 31C12, 31C30,  53C21, 39A12}

\maketitle

\begin{abstract}

In this paper we investigate the  $ L^1 $-Liouville property, underlining its connection with  stochastic completeness and other structural features of the graph. We give a characterization of the $ L^1 $-Liouville property in terms of the Green function of the graph and use it to prove its equivalence with stochastic completeness on model graphs. Moreover, we show that there exist stochastically incomplete graphs which satisfy the $ L^1 $-Liouville property and prove some comparison theorems for general graphs based on inner-outer curvatures. We also introduce the Dirichlet $L^1$-Liouville property of subgraphs and prove that if a graph has a Dirichlet $L^1$-Liouville subgraph, then it is $L^1$-Liouville itself. As a consequence, we obtain that the $ L^1$-Liouville property is not affected by a finite perturbation of the graph and, just as in the continuous setting, a graph is $ L^1$-Liouville provided that at least one of its ends is Dirichlet $ L^1$-Liouville.\\

\noindent
DiSTA,   Universit\'a dell'Insubria,  Via Valleggio 11, 22100 Como, Italy\\
 (aadriani@uninsubria.it)\\
DiSAT,   Universit\'a dell'Insubria,  Via Valleggio 11, 22100 Como, Italy\\
 (alberto.setti@uninsubria.it)\\

\end{abstract}

\section{Introduction}

The study of stochastic properties of weighted graphs, such as parabolicity, stochastic completeness and the Feller property, to mention some of the most relevant, essentially deals with properties of solutions to classical equations (Poisson and heat) involving the Laplace operator of the graph. It is of great importance because it highlights a mutual interaction between analytic and geometric properties of the graph under consideration.

This area of investigation has undergone a rapid growth in the past several years also in relation to analysis on metric spaces (\cite{F14,HK,XH11,XH14,HKS20,KL12,KLW}). In particular, the study of weighted graphs where the measure is possibly unrelated to the coefficients of the discrete Laplacian produces a very rich theory which exhibits many unexpected features, see, to mention just a few, \cite{KLW,M09,MW19,RKW11,RKW21} and the comprehensive up to date account in \cite{KLW21}.

Because of this, conditions implying parabolicity, stochastic completeness, etc. may be substantially different from those valid in the setting of Riemannian manifolds, of which graphs represent in many ways a natural discretization.

We mention, among others, the sharp logarithmic curvature lower bound for stochastic completeness, \cite[Theorem 4.11]{MW19}, the fact that  curvature, in particular the Ollivier curvature, does not control volume as well as on manifolds, \cite{AS}, (although, in this respect, we recall that Bonnet-Meyers results may be obtained assuming positive lower bounds on the Ollivier curvature, \cite{LLY11,MW19,Oll07}) and, more generally, the fact that the Laplacian of the distance function is not so tightly tied to volume growth and stochastic properties, see \cite[Example 6.1]{A}.

This originates from the fact that in the graph setting there is not a close relationship between the metric, the measure and the coefficients of the Laplace operator (weighted graphs should be really compared to weighted manifolds). At a technical  level, the discrete structure of a graph introduces unexpected difficulties, generally arising from non-locality and the lack of a chain rule.

Moreover, it should be noted that geometric objects that agree in the continuous setting are different on graphs. However, by using compatible metrics, or considering suitable structures, in many cases it has been possible to find exact analogues of results valid on manifolds (see the interesting and exhaustive surveys \cite{K15,RKW21}).

Since the (super/sub)-harmonicity of a function does not depend on the measure, it is always possible, given a non-costant (super/sub)-harmonic function $ f $, to find a sufficiently small measure $ m $ such that $ f \in L^p(V,m) $ for every $ p \in (0,\infty) $ (see \cite{M09}). On the other hand, under the assumption that a compatible intrinsic metric exists, an analogue of Karp's Liouville Theorem for $ p>1 $, \cite[Theorem 2.2]{Karp82}, was proved in \cite{HK}. There, extending a construction in \cite[Examples 4.1 and 4.2]{HJ14}, the authors also exhibit an example of a graph with a compatible intrinsic metric (which may be interpreted as a metric completeness condition on the graph) supporting a non-constant harmonic function in $ L^p $ for every $ p \in (0,1] $. This shows that a version of the Yau's Liouville theorem for $p<1$, \cite[Theorem 3]{Yau}, extended by K.-T. Sturm (\cite[Theorem 1]{Sturm94}) to the case of local Dirichlet forms, does not hold on graphs. However, it would be interesting to investigate if an $ L^p $-Liouville result for $ p \leq 1 $ can be established under suitable geometric conditions on the graph.


In this paper we consider the limit case $ p=1 $ and address the problem of the existence of non-constant, non-negative $L^1$ super-harmonic functions. We say that a graph satisfies the $ L^1$-Liouville property, or it is $ L^1$-Liouville, if every non-negative, super-harmonic function in $ L^1 $ is constant. Of course, starting from a non-parabolic graph and suitably changing the measure one can always find a non-constant, non-negative $L^1$ super-harmonic function. Thus, the problem only makes sense if one introduces suitable restrictions on the graph as in the extension of Karp's result.

Rather than assuming the existence of a compatible intrinsic metric, we consider the relationship of the $ L^1$-Liouville property with other stochastic properties, most notably stochastic completeness. Our study follows the lines of \cite{BPS,PPS} in the setting of Riemannian manifolds. The main issues originate from the, already underlined, non-locality of the Laplacian operator, which yields in turn the lack of a chain rule. Despite these remarkable differences, we are able to obtain various analogues of results valid for Riemannian manifolds. In particular, after having recalled that stochastic completeness implies the $ L^1 $-Liouville property (see \cite{HK}), we first show that these properties are equivalent for model graphs and then we exhibit examples of stochastically incomplete graphs with the $L^1$-Liouville property (Examples \ref{exmp_1} and \ref{exmp_2}) and establish comparison results in terms of inner-outer curvatures (Theorems \ref{comparison 2} and \ref{comparison 1}). In the final part of the paper we introduce the Dirichlet $L^1$-Liouville property of subgraphs and prove that if a graph has a Dirichlet $L^1$-Liouville subgraph, then it is $L^1$-Liouville itself. As in the Riemannian case, this is in contrast to what happens for stochastic completeness which holds if and only if all ends are stochastically complete (see \cite{XH11} and \cite[Theorem 4]{KL12} for a general result in the discrete setting). As a consequence, the $ L^1$-Liouville property is not affected by a finite perturbation of the graph and, just as in the continuous setting, a graph is $ L^1$-Liouville provided that at least one of its ends is Dirichlet $ L^1$-Liouville (Theorem \ref{end_thm}).

The paper is organized as follows: in Section \ref{sec:set up} we introduce some basic notations and definitions of graph theory. In Section \ref{sec:L_1} we define the $L^1$-Liouville property, we describe its relationship with stochastic completeness and, after having proved that the two concepts are equivalent for model graphs, we exhibit two examples of stochastically incomplete graphs with the $ L^1$-Liouville property. In Section \ref{sec:comp} we prove two comparison results with model graphs. They are proved by transplanting to the graph the Green function, respectively, the mean exit time, of the model. In Section \ref{sec:DirL_1} we introduce the Dirichlet $L^1$-Liouville of a subgraph and study its relationships with the $ L^1$-Liouville property of the ambient graph.


\section{Set up and basic facts}\label{sec:set up}

A graph is a triple $ G = (V,b,m) $, where $ V $ is the set of nodes of the graph, $ b: V \times V \to [0,\infty) $ is a symmetric function vanishing on the diagonal and $ m : V \to (0,\infty) $ is a measure of full support. Consider two nodes $ x, y \in V $. When $ b(x,y) \neq 0 $ we say that $ x $ and $ y $ are neighbors and we write $ x \sim y $. A path is a (possibly infinite) sequence of neighbors and we say that a graph is connected if, for every $x,y \in V $, there exists a finite path connecting $ x $ and $ y $, that is $ \exists \, \ x=x_0 \sim x_1 \sim ... \sim x_{n-1} \sim x_{n} = y $ for some $ x_1,...,x_{n-1} \in V $.
In this case $n$ is the (combinatorial) length of the path and we define the (combinatorial) distance $d(x,y)$ between $x$ and $y$   to be the infimum of the length of paths connecting $x$ and $y$.

We say that a graph is locally finite if, for every $ x \in V $, $ |\{ y \, : \, b(x,y) \neq 0 \}| < \infty $, where $ | \cdot | $ denotes the cardinality of a set. In particular, local finiteness implies that, for every $ x \in V $,
\begin{equation*}
\frac 1{m(x)}\sum_{y \in V} b(x,y) < \infty.
\end{equation*}
The quantity in the above equation is called degree at $ x $ and  it is denoted by $ \text{Deg}(x) $.
From now on, we always assume graphs to be connected and locally finite.

We denote by $ C(V) $ the set of all functions $ f : V \to \mathbb{R} $. Under the assumption of local finiteness, it is immediate to see that the Laplacian $ \Delta : C(V) \to C(V) $ of the graph, defined by the formula
\begin{equation*}
\Delta f(x) = \frac{1}{m(x)} \sum_{y \in V} b(x,y) \left( f(x)-f(y) \right),
\end{equation*}
is well defined for every $ f \in C(V) $.

A subgraph $ N $ of $ G =(V,b,m) $ is a triple $ N= (W, b_{|W \times W}, m_{|W}) $, where $ W \subset V $. We denote by $ \text{int\,}N = \{ y \in W \, : \, \nexists \, x \in V \setminus W \text{ such that } b(x,y) \neq 0 \} $ the interior of $ N $ and by $ \partial N = N \setminus \text{int\,}N $ its (inner) boundary. Given a fixed vertex $ x_0 \in V $ and $ r \geq 0 $, we write $ B_r(x_0) = B_r = \{ y \in V \, : \, d(x_0,y) \leq r \} $ and $ S_r(x_0) = S_r =\{ y \in V \, : \, d(x_0,y) = r\}$.

The inner and outer curvatures $k_{\pm}(x) $ of a graph at $ x \in S_r(x_0) $ are defined by
\begin{equation*}
k_{\pm}(x) = \frac{1}{m(x)} \sum_{y \in S_{r \pm 1}} b(x,y).
\end{equation*}
When $ k_{\pm} $ are spherically symmetric functions, i.e. when $ k_{\pm}(x) = k_{\pm} (x') $ for every $ x,x' \in S_r(x_0) $, for every $ r \geq 0 $, we say that the graph is weakly spherically symmetric, or that it is  a model with root $ x_0 $. Following \cite{KLW}, we set
\begin{equation*}
\partial B(r) = \sum_{x \in S_r} \sum_{y \in S_{r+1}} b(x,y).
\end{equation*}
The heat kernel $ p_t(x,y) $ is the minimal positive fundamental solution of the continuous time heat equation on $G$, namely,
\[
\begin{cases}
(\partial_t+\Delta_y)\, p_t(x,y)=0 & \forall x,y \in V\, \text{ and }\, t>0\\
p_0(x,y)= \frac{\delta_x(y)}{m(x)},\\
\end{cases}
\]
and it can be obtained via exhaustion of the graph (see, for example, \cite{RKW1}). In particular, the heat kernel is the monotone limit of the Dirichlet heat kernels, which we denote by $ p_t^r(x,y) $, on the exhaustion $ \{ B_r \}_{r \geq 0} $, which satisfy the initial value problem in  $\mathrm{int\,} B_r$ and Dirichlet boundary conditions, i.e., $p^r_t(x,y)=0$ for all $t>0$ if $x$ or $y$ are in $ \partial B_r$.
It is well-known that
\begin{equation*}
\sum_{y \in V}  p_t(x,y) m(y) \leq 1
\end{equation*}
and we say that the graph (or, more precisely, the Browninan motion driven by the Laplace operator) is stochastically complete if, for every $ x \in V $ and every $ t > 0 $,
\begin{equation*}
\sum_{y \in V}  p_t(x,y) m(y) = 1.
\end{equation*}

Recall also that the graph  is said to be non-parabolic (or transient) if it admits a Green kernel, that is, a minimal positive fundamental solution to the Poisson equation
\[
\Delta_y g(x,y)= \frac{\delta_x(y)}{m(x)}.
\]
The Green kernel can be obtained via time integration of $ p_t(x,y) $ as follows:
\begin{equation*}
g(x,y) = \int_{0}^{\infty} p_t(x,y) dt,
\end{equation*}
so that transience is equivalent to the finiteness of the integral on the right hand side. Similarly, the Dirichlet Green kernels are given by
\begin{equation*}
g_r(x,y) = \int_{0}^{\infty} p_t^r(x,y) dt,
\end{equation*}
and the Green kernel of $G$ is then the monotone limit of the Dirichlet Green kernels of an exhaustion.


\section{The $ L^1$-Liouville property}\label{sec:L_1}

In this section we first define the $ L^1 $-Liouville property, then give some characterization and preliminary results about it in the context of general graphs.

\begin{defn}
We say that a graph $ G=(V,b,m) $ satisfies the $ L^1 $-Liouville property, or, shortly, that it is $ L^1 $-Liouville, if every non-negative super-harmonic function in $ L^1(V,m) $ is constant, that is, every non-negative function $ u \in L^1(V,m) $ such that
$$
\Delta u \geq 0
$$
is constant.
\end{defn}

Recalling that parabolicity is equivalent to the fact that every non-negative super-harmonic function is constant, all parabolic graphs are trivially $ L^1$-Liouville, so we can concentrate our study to non-parabolic ones.
We start with the following characterization of the $ L^1 $-Liouville property in terms of the (non)integrability of the Green function $ g $ of the graph. This is the analogue of a result in \cite{Gri} and is the content of Theorem 1.7 in \cite{HK}. We provide a slightly different proof based on the strong minimum principle according to which if a non-constant function $ u $ satisfies
\begin{equation*}
\begin{cases}
\Delta u \geq 0 & \text{ on int\,} D  \\
u\geq 0 & \text{ on } \partial D
\end{cases}
\end{equation*}
on a finite subgraph $ D $, then $ u > 0 $ on int\,$D$ (see, for example, \cite{D84,Gri18,Web10}).

\begin{thm}\label{non-int}
A graph $ G=(V,b,m) $ is $ L^1 $-Liouville if and only if, for some (any) $ x \in V $, $ g(x, \cdot) \notin L^1(V,m) $, that is,
\begin{equation*}
\sum_{y \in V} g(x,y) m(y) = \infty.
\end{equation*}
\end{thm}

\begin{proof}
$ \Rightarrow $ This implication is obvious: $ g $ is a super-harmonic, non-negative function. Since $ g $ is non-constant and the graph is $ L^1 $-Liouville, it follows that $ g(x, \cdot) \notin L^1(V,m) $ for every $ x \in V $.

$ \Leftarrow $ Let $ u : V \to \mathbb{R} $ be a non-negative, non-constant, super-harmonic function and let $ x_0 \in V $ be a fixed vertex. We need to show that $ g(x_0, \cdot) \notin L^1(V,m) $ implies $ u \notin L^1(V,m) $. There exists $ C > 1 $ such that $g(x_0,x_0) \leq C u(x_0) $. Note that the function $ v(x) := C u(x)-g_r(x_0,x)  $, where $ g_r $ is defined as above, satisfies the following system:
\begin{equation*}
\begin{cases}
\Delta v(x) \geq 0 & \text{ for all } x \in \mathrm{int}\, B_r(x_0) \setminus \{ x_0 \} \\
v(x) \geq 0 & \text{ for all } x \in \partial B_r(x_0) \cup \{ x_0 \}.
\end{cases}
\end{equation*}
It follows by the maximum principle that $ v \geq 0 $ on $ B_r(x_0) $ for every $ r \geq 1 $.
By passing to the limit as $ r \to \infty $ and using the fact that $ C > 1 $, we obtain $ C u \geq g(x_0, \cdot) $, so that $ u \notin L^1(V,m) $ by the assumed non-integrability of $ g(x_0,\cdot) $.
\end{proof}

Recalling the connection between the Green function $ g $ of a graph and its heat kernel, we obtain the following immediate corollary (see \cite[Theorem 1.7]{HK}).

\begin{cor}\label{stoc_impl_L1}
Every stochastically complete graph $ G=(V,b,m) $ satisfies the $ L^1 $-Liouville property.
\end{cor}

\begin{proof}
Indeed, by applying Tonelli's Theorem,
\begin{equation*}
\begin{split}
\sum_{y \in V} g(x_0,y) m(y) & = \sum_{y \in V} \int_{0}^{\infty} p_t(x_0,y) dt \, m(y) \\
& = \int_{0}^{\infty} \sum_{y \in V} p_t(x_0,y) m(y) \, dt \\
& = \int_{0}^{\infty} 1 dt = \infty
\end{split}
\end{equation*}
and the graph is $ L^1 $-Liouville by Theorem \ref{non-int}.
\end{proof}

\subsection{The $ L^1 $-Liouville property and stochastic completeness of model graphs}\label{ssec:L_1}

In this subsection we want to show that for a model graph the two properties of stochastic completeness and $ L^1 $-Liouville actually coincide, which represents the analogue of a result in \cite{BPS}. It is well known, see \cite[Theorem 1]{KLW}, that if a graph is a model then the heat kernel is a spherically symmetric function, meaning that $ p_t(x_0,x) = p_t(r) $ for every $ x \in S_r(x_0) $ and every $ r \in \mathbb{N}_0 $. This implies that, on model graphs, the function $ g(x_0,\cdot) $ is a spherically symmetric function.

Indeed, a straightforward computation shows that the Green function of a model graph centered at the root $x_0$ is given by
\begin{equation}\label{Green model}
g(x_0,x) = g(r) = \sum_{k=r}^{\infty} \frac{1}{\partial B(k)},
\end{equation}
for every $ x \in S_r(x_0) $ and $ r \geq 0 $.
Moreover, when $ g \equiv \infty $, i.e. when the graph is parabolic, or recurrent, then it is trivially stochastically complete and, therefore, satisfies the $ L^1 $-Liouville property. We are then interested in the case where $ G $ is a model graph and $ g \neq \infty $, that is, $ \sum_k \frac{1}{\partial B(k)} < \infty $.

We also recall the following result (see \cite[Theorem 5.10]{XH11} for the unweighted case and \cite[Theorem 5]{KLW} and \cite[Theorem 4.8]{RKW11} for the general case):

\begin{thm}\label{stochcompletenessmodels}
A model graph ${G}$ is stochastically complete if and only if the series
\begin{equation*}\label{sum}
\sum_{k=0}^{\infty} \frac{m(B_k)}{\partial B(k)}
\end{equation*}
diverges.
\end{thm}

With this preparation we have the following theorem.

\begin{thm}
Let $ G=(V,b,m) $ be a model graph. Then $ G $ is $ L^1 $-Liouville if and only if it is stochastically complete.
\end{thm}

\begin{proof}
We may assume that $ G $ is not parabolic and therefore that $ \sum_k \frac{1}{\partial B(k)} < \infty $.

Changing the order of summation and using Equation \eqref{Green model}, we see that
\begin{equation*}
\sum_{r=0}^{\infty}g(r)m(S_r) = \sum_{r=0}^{\infty} \sum_{k=r}^{\infty} \frac{m(S_r)}{\partial B(k)} = \sum_{k=0}^{\infty} \sum_{r=0}^{k} \frac{m(S_r)}{\partial B(k)}= \sum_{k=0}^{\infty} \frac{m(B_k)}{\partial B(k)},
\end{equation*}
so that the condition for the $ L^1 $-Liouville property to hold coincides with the one for the stochastic completeness. This completes the proof.
\end{proof}

At this point, one could wonder whether stochastically incomplete graphs satisfying the $ L^1 $-Liouville property exist, or if a converse of Corollary \ref{stoc_impl_L1} is true. We give two examples to show that in fact there exist graphs which are stochastically incomplete, but satisfy the $ L^1 $-Liouville property. The first one is an analogue of an example proposed in \cite{BPS} and is similar to other examples in \cite{HK,XH11}, while the second is based on the notion of antitree, which we will briefly introduce after the first example.

\begin{exmp}\label{exmp_1}
We start by considering two graphs $ M_1=(V_1,b_1,m_1) $ and $ M_2=(V_2,b_2,m_2) $.

We make the following assumptions on $ M_1 $ and $ M_2 $:

\begin{enumerate}
\item $ m_1(M_1) = \sum_{x \in V_1} m_1(x) = \infty $, that is, $ M_1 $ has infinite volume.
\item $ M_2 $ is stochastically incomplete, for instance, $M_2$ is a model graph with $ \sum_{k=0}^{\infty} \frac{m_2(B_k)}{\partial B(k)} < \infty $.
\end{enumerate}

We define a new graph $ M =(V,b,m) $ by gluing the two graphs $ M_1 $ and $ M_2 $ at a single vertex, that is $ V:= V_1 \cup V_2 $, $ b_{|\left( V_i \times V_i \right)} := b_i $, $ m_{|V_i} := m_i $ and there exist two and only two vertices $ x_1 \in V_1 $ and $ x_2 \in V_2 $ such that $ b(x_1,x_2) >0 $. Since $ M_2 $ is stochastically incomplete, it turns out that $ M $ is stochastically incomplete as well (see \cite{XH11,KL12}) and, therefore, it is not parabolic. We want to show that it is possible to define a conformal change of the measure so that stochastic incompleteness is preserved and $ M $ becomes $ L^1 $-Liouville.

Let $ \tilde{m}(y) = \lambda^{2}(y)m(y) $ for all $ y \in V $, where $ \lambda(y)=1 $ if $ y \in V_2 $ and
$\lambda^{2}(y)g(x,y) \geq 1 $ for $ y \in V_1 $, where $ g $ is the Green function of the graph $ M=(V,b,m) $. We claim that this conformal change of the measure preserves the Green function, that is $ \tilde{g} = g $, where $ \tilde{g} $ is the Green function of the graph $\tilde{M}=(V, b,\tilde{m}) $. Indeed, using the trivial fact that the Laplacian on $ \tilde{M} $ is $ \tilde{\Delta} = \frac{1}{\lambda^{2}} \Delta $ and the property of the Green function,
\begin{equation*}
\begin{split}
f(x) & = \sum_{y \in V} \Delta_y g(x,y)f(y)m(y) \\
& = \sum_{y \in V} \frac{\Delta_y g(x,y)}{\lambda^{2}(y)}f(y)\lambda^{2}(y)m(y) \\
& = \sum_{y \in V} \tilde{\Delta}_y g(x,y) f(y) \tilde{m}(y),
\end{split}
\end{equation*}
proving our claim. Note that, since $ \lambda^{2}=1 $ on $ V_2 $, $ \tilde{\Delta}= \Delta^{M_2} $ on $V_2\setminus \{x_2\}, $
and it follows easily that $ \tilde{M} $ is stochastically incomplete by the weak Omori-Yau maximum principle, \cite[Theorem 2.2]{XH11}.

We conclude our example by showing that $ \tilde{M} $ is $ L^1 $-Liouville. Indeed,
\begin{equation*}
\sum_{y \in V} \tilde{g}(x,y) \tilde{m}(y)= \sum_{y \in V} g(x,y) \lambda^{2}(y)m(y) \geq \sum_{y \in V_1} m(y) = \infty
\end{equation*}
and $ \tilde{M} $ is $ L^1 $-Liouville by Theorem \ref{non-int}.
\end{exmp}

Antitrees have been used to construct  various examples in \cite{CLMP,DK88,KLW,Web10,RKW11}. We use this notion to produce a graph which is stochastically incomplete, $ L^1 $-Liouville and has only one end, i.e., the complement of any  finite subset has only one unbounded component. Note that the graph constructed in the previous example has at least two ends, corresponding to the two different graphs glued together.

\begin{defn}
A graph $  G=(V,b,m) $ is called antitree if $ m \equiv 1 $, $ b(x,y) \in \{ 0,1\} $ and $ k_{+}(r)=|S_{r+1}| $ for all $ r \in \mathbb{N}_0 $. This last requirement is equivalent to saying that every vertex in $ S_r $ is connected to every vertex in $ S_{r+1} $.
\end{defn}

\begin{remark}
For our purposes we also require that there are no internal connections on each sphere. Such spherical connections would play a crucial role in the calculation of some notions of curvature such as the Ollivier curvature, see, for instance, \cite{CLMP}, where the authors require that every vertex is connected with every other vertex in the same sphere.
\end{remark}

\begin{exmp}\label{exmp_2}
Let $ A=(V,b,m) $ be an antitree with root $ x_0 $ and no spherical connections. We assume $ A $ to be stochastically incomplete. Since $ A $ is a model,  using the definition of an antitree, this amounts to the fact that
\begin{equation*}
\sum_{k=0}^{\infty} \frac{m(B_k)}{\partial B(k)}= \sum_{k=0}^{\infty} \frac{|B_k|}{|S_k||S_{k+1}|} < \infty.
\end{equation*}
Let $ \gamma $ be an infinite path $ \{x_i\}_{i=0}^{\infty} $ in $ A $, where $ x_i \in S_i $ for all $ i $.
We now apply a conformal change of the measure to guarantee the $ L^1 $-Liouville property, while preserving stochastic incompleteness.  Let $ \lambda:V \to \mathbb{R} $ such that $ \lambda \geq 1 $, $ \lambda(x)=1 $ for all $ x \notin \gamma $ and $ \lambda(x) \geq g(x_0,x)^{-\frac{1}{2}} $ for all $ x \in \gamma $. We now consider the graph $ \tilde{A}=(\tilde{V},\tilde{b},\tilde{m}) $, where $ \tilde{V}=V $, $ \tilde{b}=b $ and $ \tilde{m}=\lambda^2 m $.

As in the above example $ \tilde{g}(x_0,x) = g(x_0,x) $, and it is easily verified that $ \tilde{A} $ has the $ L^1 $-Liouville property. Indeed,
\begin{equation*}
\sum_{y \in \tilde{V}} \tilde{g}(x_0,y) \tilde{m}(y) \geq \sum_{y \in \gamma} g(x_0,y) \lambda^2(y) m(y) \geq \sum_{y \in \gamma} m(y) = \infty
\end{equation*}
and the graph is $ L^1 $-Liouville.

To conclude we need to show that the graph is stochastically incomplete or, equivalently, that it does not satisfy the weak Omori-Yau maximum principle (see \cite[Theorem 2.2]{XH11}).

Letting $ r(x) = d(x_0,x) $ and $ a_l = \frac{|B_l|}{|S_l||S_{l+1}|} $, we define $f^*=\sum_{l=0}^{\infty} a_l<+\infty$ and
\begin{equation*}
f(x)= \begin{cases}
\sum_{l=0}^{r(x)-1} a_l & \text{ } x \notin \gamma \\
f^*-\epsilon & \text{ for } x = x_i, \text{ for all } i \geq n-1 \\
0 & \text{ otherwise},
\end{cases}
\end{equation*}
with $ \epsilon $ and $ n $ to be determined. Then $f$ is bounded above with $\sup f=f^*$ and, if $ x \in S_r \setminus \gamma $, we have:
\begin{equation*}
\begin{split}
\tilde{\Delta}f(x) & = \sum_{y \in S_{r-1}} \left( f(x) - f(y) \right) + \sum_{y \in S_{r+1}} \left( f(x) - f(y) \right) \\
& = \left( |S_{r-1}|-1 \right) a_{r-1} + f(x)-f(x_{r-1}) + \left( 1-|S_{r+1}| \right) a_{r} + f(x) - f(x_{r+1}).
\end{split}
\end{equation*}
Note that
\begin{equation*}
\begin{split}
\left( |S_{r-1}|-1 \right) a_{r-1} + \left( 1-|S_{r+1}| \right) a_{r}
& =
\frac{|B_{r-1}|}{|S_{r}|} - a_{r-1}+ a_{r} - \frac{|B_{r}|}{|S_{r}|} \\
& = -1 + a_{r} - a_{r-1},
\end{split}
\end{equation*}
so that, inserting into the above expression,  we get
\[
\tilde{\Delta} f(x)  = -1 + a_{r} - a_{r-1} +  2 \sum_{l=0}^{r-1} a_l - 2 f^{*} + 2 \epsilon.
\]
Now we choose $ \epsilon > 0 $ and $ n > 1 $ such that
\begin{enumerate}
\item $ -1 + 3 \epsilon < 0 $,
\item  $ a_{r}-a_{r-1} < \epsilon $ for all $ r \geq n-1 $,
\item $ f^{*} - \sum_{l=0}^{n-2} a_l < \epsilon $,
\end{enumerate}
and let $ \alpha = \sum_{l=0}^{n-1} a_l $, so that \[\Omega_{\alpha}= \{x\,:\, f(x)>\alpha\}= B_{n-2}^{c}(x_0) \setminus \gamma.\]
 It follows that if $ x \in \Omega_{\alpha} $, then $r(x)\geq n-1$ and
\begin{equation*}
\begin{split}
\tilde{\Delta} f(x) & \leq -1 + a_{r} - a_{r-1} +  2 \epsilon
 \\
& < -1 + 3 \epsilon < 0,
\end{split}
\end{equation*}
showing that $f$ violates the weak maximum principle at infinity and therefore $\tilde{A}$ is stochasticastically  incomplete, as claimed.
\end{exmp}

\section{Comparison theorems}\label{sec:comp}

In this section we state and prove two comparison results with model graphs based on the notions of stronger/weaker curvature growth (see \cite{A,AS,KLW,RKW2}, where they are used to prove comparison theorems concerning volume growth, stochastic completeness and the Feller property). The first theorem is obtained by transplanting the Green function of the comparison model. The second uses the characterization of the $L^1$-Liouville property in terms of mean exit time and, again, a transplantation technique.

\begin{defn}
Let $ \tilde{G} =(\tilde{V},\tilde{b},\tilde{m}) $ be a model. We say that a graph $ G =(V,b,m) $ has stronger curvature growth than $ \tilde{G} $  outside   a finite set if there exists a vertex $ x_0 \in V $  and $R\geq 0$ such that
\begin{equation*}
k_+(x) \geq \tilde{k}_+(r) \text{ and } k_-(x) \leq \tilde{k}_-(r)
\end{equation*}
for all  $ x \in S_r(x_0) $ in $ G $ with $ r \geq R$.

We say that $ G $ has weaker curvature growth than $ \tilde{G} $ outside of a finite set if there exists $ x_0 \in V $ such that opposite inequalities hold.

In the case where $ R = 0 $ we simply say that $ G $ has stronger/weaker curvature growth than $ \tilde{G} $.
\end{defn}

\begin{thm}\label{comparison 2}
Let $ G=(V,b,m) $ be a weighted graph and let $ \tilde{G} = (\tilde{V},\tilde{b},\tilde{m}) $ be a model graph with root $ o $. Suppose that $ G $ has stronger curvature growth than $ \tilde{G} $ outside a finite set and $ m(S_r(x_0)) \leq C \tilde{m} (S_r(o)) $ for some $ C > 0 $ and all $ r $ large enough. If $ \tilde{G} $ is not $ L^1 $-Liouville, then $ G $ is not $ L^1 $-Liouville.
\end{thm}

\begin{remark}
A result in \cite{AS} shows that if $ G $ has  stronger curvature  than a model graph $ \tilde{G} $ outside   a finite set then there exists a constant $ c > 0 $ such that $ m(S_r(x_0)) \geq c \tilde{m} (S_r(o)) $. The  assumption that $ m(S_r(x_0)) \leq C \tilde{m} (S_r(o)) $ amounts to asking that the volume growth of $G$ is controlled from below and from above by the volume growth of the comparison model graph. This is due to the fact, already stressed in the previous section, that $ L^1 $-Liouville property depends on the behaviour of the volume at infinity.
\end{remark}

\begin{proof}
Assume that $ \tilde{G} $ is not $ L^1 $-Liouville, so that by Theorem \ref{non-int}
\begin{equation*}
\sum_r \tilde{g}(r) \tilde{m} \left( S_r(o) \right) < \infty.
\end{equation*}
In order to negate the $ L^1 $-Liouville property, we need to find a function which is integrable, non-negative, super-harmonic and non-constant.
For  $ x \in S_r(x_0) $ define
\begin{equation*}
v(x)  := \tilde{g}(r),
\end{equation*}
and note that, by the explicit expression \eqref{Green model} of $\tilde{g}$, $ v $ is a decreasing radial function.
Assuming that $G$ has stronger curvature than $\tilde{G}$ if  $r(x)\geq R$ for some $ R > 0 $, for all such $x$ we have
\begin{equation*}
\begin{split}
\Delta v(x) & = k_+(x) (v(r) - v(r+1)) + k_-(x) (v(r) - v(r-1)) \\
& \geq \tilde{k}_+(r) \left( \tilde{g}(r) -  \tilde{g}(r+1) \right) + \tilde{k}_-(r) \left( \tilde{g}(r) -  \tilde{g}(r-1) \right) = \Delta \tilde{g}(r) \geq 0,
\end{split}
\end{equation*}
showing that  $v$ is non-negative and super-harmonic in $B_R(x_0)^c$. It follows that the function
$u=\min\{v(x), \tilde{g}(R+1)\}$  is  non-costant, non-negative, and super-harmonic on $G$.

Moreover, by increasing $ C $, we may assume that $ m(S_r(x_0)) \leq C \tilde{m}(S_r(o)) $ for every $ r \geq R $, so that
\begin{equation*}
\sum_{x \in V} u(x) m(x)\leq C+ \sum_{r\geq R} \tilde{g}(r) m\left( S_r(x_0) \right) \leq C \sum_r \tilde{g}(r) \tilde{m} \left( S_r (o) \right) < \infty,
\end{equation*}
as required to conclude that $ G $ is not $ L^1 $-Liouville.
\end{proof}

The following is a non-trivial application of the previous theorem with $ C = 1 $.

\begin{exmp}
We define $ \tilde{G}=(\mathbb{N}_0,\tilde{b},\tilde{m}) $ and $ G=(\mathbb{N}_0,b,m) $ by
\begin{equation*}
\tilde{m}(r) = \begin{cases}
1 & r = 0 \\
e^{-r}+2 & r \geq 1,
\end{cases}
\quad \tilde{b}(r,r+1)= (r+1)^3 \text{ for all } r \geq 0 \text{ and }
\end{equation*}
\begin{equation*}
m(r) = 2 \frac{e^r}{e^r+1},
\quad
b(r,r+1)= \frac{2 e^{r+1}}{e^{r+1}+1} \cdot \frac{(r+1)^3}{e^{-(r+1)}+2} \text{ for all } r \geq 0.
\end{equation*}
Straightforward computations give
\begin{equation*}
\tilde{k}_+(r) = \frac{(r+1)^3}{e^{-r}+2} \text{ for all } r \geq 1
\end{equation*}
and
\begin{equation*}
k_+(r) = \frac{(r+1)^3}{e^{-(r+1)}+2} \frac{e(e^r+1)}{e^{r+1}+1} \text{ for all } r \geq 0,
\end{equation*}
so that $ \tilde{k}_+ (r) \leq k_+(r) $ for all $ r \geq 1 $.
Moreover,
\begin{equation*}
\tilde{k}_-(r) = k_-(r) = \frac{r^3}{e^{-r}+2},
\end{equation*}
so that $ \tilde{G} $ has weaker curvature growth than $ G $ outside a finite set.

Note that, since
\begin{equation*}
\sum_{r=0}^{\infty} \frac{\sum_{i=0}^{r} e^{-i}}{(r+1)^3} < \infty
\end{equation*}
and
\begin{equation*}
\sum_{r=0}^{\infty} \frac{2(r+1)}{(r+1)^3} < \infty,
\end{equation*}
$ \tilde{G} $ is stochastically incomplete, hence, not $ L^1 $-Liouville.

Finally, using the fact that $ \tilde{m}(r) \geq m(r) $ for all $ r \geq 1 $ and Theorem \ref{comparison 2}, we conclude that $ G $ is not $ L^1 $-Liouville.

\end{exmp}

\subsection{Mean exit time and $ L^1 $-Liouville}\label{ssec:mean}

Let $E_r $ be the solution of the following problem:
\begin{equation*}
\begin{cases}
\Delta E_r = 1 & \text{ in } \mathrm{int}\,B_{r}(x_0) \\
E_r=0 & \text{ on } \partial B_r(x_0).
\end{cases}
\end{equation*}
Note that $ E_r $ either diverges at every point, or it converges to a function $ E $ satisfying $ \Delta E \equiv 1 $.
Indeed, by the maximum principle, $ E_r $ is positive and strictly increasing so that $ \lim_r E_r(x) = E(x) $ exists, finite or infinite. Assume that $ E(x_0) < \infty $ for some $ x_0 $. Since $ \Delta E_r(x_0) = 1 $, rearranging we get
\[
m(x_0) + \sum_{y} b(x_0,y) E_r(y) = E_r(x_0) \sum_{y} b(x_0,y)
\]
whence, letting $ r \to \infty $ and using the monotone convergence theorem, we deduce that the limit of the left hand side is finite, so in particular $ E(y)< \infty $ for all $ y \sim  x_0 $ and $ \Delta E(x_0)=1 $. A connectedness argument then shows that $ E(y) < \infty $ for every $ y $ and $ \Delta E = 1 $.

We also note that we have an explicit representation of $ E_r $ in terms of the Dirichlet Green's kernel $ g_r(x_0,x) $ of $B_r(x_0)$,
\begin{equation*}
E_r(x)= \sum_{y \in B_{r}(x_0)} g_r(x,y)m(y).
\end{equation*}
In particular
\begin{equation*}
E_r(x) \nearrow E(x) = \sum_{y \in V} g(x,y)m(y),
\end{equation*}
with $ g $ being the Green kernel of $ G $, so that the graph $ G $ is not $ L^1$-Liouville if and only if $ E(x) < \infty $ for some/every $ x \in G $.

Before stating and proving the main theorem of this subsection, we recall the following result, which was proved in \cite{KLW}.

\begin{thm}[Theorem 6 in \cite{KLW}]\label{comp_stoc}
If a weighted graph $ G=(V,b,m) $ has stronger (respectively, weaker) curvature growth than a weakly spherically symmetric graph $ \tilde{G} = (\tilde{V}, \tilde{b}, \tilde{m}) $ which is stochastically incomplete (respectively, complete), then $ G $ is stochastically incomplete (respectively, complete).
\end{thm}

\begin{thm}\label{comparison 1}
Let $ \tilde{G} $ be a model graph and let $ G $ be a graph.
\begin{itemize}
\item[1)] If $ G $ has weaker curvature growth than $ \tilde{G} $ and $ \tilde{G} $ is $L^1$-Liouville, then $ G $ is $ L^1$-Liouville.
\item[2)] If $ G $ has stronger curvature growth than $ \tilde{G} $ and $ \tilde{G} $ is not $ L^1 $-Liouville, then $ G $ is not $ L^1 $-Liouville.
\end{itemize}
\end{thm}

\begin{proof}
The first statement follows immediately by combining the equivalence between $ L^1 $-Liouville property and stochastic completeness on model graphs described in  Theorem \ref{comp_stoc} and Corollary \ref{stoc_impl_L1}.

To prove the second statement, note that, by the above considerations, we only need to show that $ E(x) = \sum_{y \in G} g(x,y)m(y) < \infty $ for some $ x \in G $. To do so, we first notice that the curvature assumption implies that, for every $R$, $\mathrm{int}\, B_R(x_0)=B_{R-1}(x_0)$ and $\partial B_R(x_0)=S_R(x_0)$. Next, consider the spherically symmetric function on $ \tilde{G} $ given by
\begin{equation*}
F_R(r) = \sum_{k=r}^{R-1} \frac{\tilde{m}(B_k)}{\partial \tilde{B}(k)},
\end{equation*}
where, as usual, we set $ \partial \tilde{B}(k) = \tilde{k}_+(k) \tilde{m}(S_k) $ for every $ k \geq 0 $. By the assumption of stochastic incompleteness and Theorem \ref{stochcompletenessmodels} we have that the function
\begin{equation*}
F(r)= \sum_{k=r}^{\infty} \frac{\tilde{m}(B_k)}{\partial \tilde{B}(k)}
\end{equation*}
is finite for every $ r $ and by monotonicity we get $ F_R \nearrow F $.
Consider now the transplanted function $ F_R(r(\cdot)) $ on $ G $ and note that the curvature condition yields
\begin{equation*}
\begin{split}
\Delta {F}_R(r(x_0)) &= \frac{1}{m(x_0)} \sum_{y \in S_1(x_0)} \left( F_R(r(x_0))-F_R(r(y)) \right) b(x_0,y) \\
& = k_+(x_0) \frac{\tilde{m}(o)}{\tilde{k}_+(o) \tilde{m}(o)} \geq 1
\end{split}
\end{equation*}
and, for  $ x \in S_r(x_0) $,  $ 0<r \leq R-1 $,
\begin{equation*}
\begin{split}
\Delta F_R(r(x)) &= \frac{1}{m(x)} \sum_{y \in S_{r+1}(x_0)} \left( F_R(r)-F_R(r+1) \right) b(x,y) \\ & \quad + \frac{1}{m(x)} \sum_{y \in S_{r-1}(x_0)} \left( F_R(r)-F_R(r-1) \right) b(x,y) \\
& = k_+(x) \frac{\tilde{m}(B_r)}{\tilde{k}_+(r) \tilde{m}(S_r)} - k_-(x) \frac{\tilde{m}(B_{r-1})}{\tilde{k}_+(r-1) \tilde{m}(S_{r-1})}.
\end{split}
\end{equation*}
Writing $ \tilde{m}(B_r) = \tilde{m}(B_{r-1}) + \tilde{m}(S_r) $, using the fact that (see, for example, \cite{KLW})
\begin{equation*} \tilde{k}_+(r-1) \tilde{m}(S_{r-1}) = \tilde{k}_-(r) \tilde{m}(S_r) = \frac{\tilde{k}_-(r) \partial \tilde{B}(r)}{\tilde{k}_+(r)}
\end{equation*}
and the hypothesis of stronger curvature growth, we conclude that, for every $ x \in S_r(x_0) $, $ 0<r \leq R-1 $,
\begin{equation*}\label{ineq}
\begin{split}
\Delta F_R(r(x)) & = \frac{\tilde{m}(B_{r-1})}{\partial \tilde{B}(r)} \left[ k_+(x) - k_-(x) \frac{\tilde{k}_+(r)}{\tilde{k}_-(r)} \right] + k_+(x) \frac{\tilde{m}(S_r)}{\partial \tilde{B}(r)} \\
& \geq \frac{\tilde{k}_+(r) \tilde{m}(S_r)}{\partial \tilde{B}(r)} = 1.
\end{split}
\end{equation*}
Summing up we have that the function $ F_R(r(\cdot))-
E_R(\cdot) $ satisfies
\begin{equation*}
\begin{cases}
\Delta \left( F_R(r(\cdot)) - E_R(\cdot) \right) \geq 0 & \text{ in }\mathrm{int}\,B_{R}(x_0) \\
F_R(r(\cdot))- E_R(\cdot)=0 & \text{ on } \partial B_R(x_0).
\end{cases}
\end{equation*}
By the maximum principle we then conclude that $ F_R(r(\cdot)) \geq E_R(\cdot) $ over $ B_R(x_0) $, so that, by passing to the limit as $ R \to \infty $, we have that $ \infty > F(r(\cdot)) \geq E(\cdot) $ and $ G $ is not $ L^1 $-Liouville.
\end{proof}

\begin{remark}
Note that since stochastic completeness/incompleteness of a (model) graph is not affected by finite sets perturbations (see, \cite{XH11, KL12}), the conclusions of Theorem~\ref{comparison 1} hold if the curvature assumptions hold outside a finite set.

Note also that Statement 2 of the above theorem is valid under the assumption of stronger curvature growth as in \cite[Definition 3.1]{A}.
\end{remark}

\section{The Dirichlet $L^1$-Liouville property}\label{sec:DirL_1}

In this section we introduce the concept of
$ \mathcal{D}$-$L^1$-Liouville property (Dirichlet Liouville property)
of infinite subgraphs modeled on the construction carried out in \cite[Section 3]{PPS}  in the setting of Riemannian manifolds. After providing a useful characterization, we prove a criterion for the validity of the $ L^1 $-Liouville property and deduce that a graph is $ L^1$-Liouville if and only if at least one of its ends is
$ \mathcal{D}$-$L^1$-Liouville. This is in contrast to what happens for stochastic completeness, which holds provided that every end of the graph is stochastically complete (see \cite{KL12,XH11}, where ideas similar  to those one in \cite{BB09} are developed in the context of weighted graphs).

Let $ N \subset G $ be an infinite, connected subgraph of $ G $.
Consider an exhaustion by finite sets $ \{ \Omega_r \} $ of $\mathrm{int\,} N $. Given $ x_0 \in \mathrm{int\, }N $, we denote by $ ^{\mathcal{D}}g_r $  the positive solution of the following system:
\begin{equation*}
\begin{cases}
\Delta {}^{\mathcal D}\! g_r (x_0,y) = \frac{\delta_{x_0}(y)}{m(x_0)} & \text{ in int\,} \Omega_r \\
{}^{\mathcal D}\! g_r(x_0,y) = 0 & \text{ on } \partial \Omega_r.
\end{cases}
\end{equation*}
Exactly as seen before, the sequence $ {}^{\mathcal D}\! g_r $ is increasing and converges to a function $ ^{\mathcal{D}}g $ which is the Dirichlet Green function on $ N $.

We are now ready to give the following definition.

\begin{defn}
We say that a subgraph $ N \subset G $ is $\mathcal{D}$-$L^1$-Liouville if every non-negative function in $ L^1(N) $ which satisfies
\begin{equation*}
\begin{cases}
\Delta u \geq 0 & \text{ in int\,} N \\
u = 0 & \text{ on } \partial N
\end{cases}
\end{equation*}
vanishes identically.
\end{defn}

In order to give a characterization of the  $\mathcal{D}$-$L^1$-Liouville property of a subgraph we introduce the notion of Dirichlet Mean Exit Time, which will be crucial for our derivations.

\begin{defn}
We denote by ${}^{\mathcal D}\! E $ the minimal positive solution of the system
\begin{equation}\label{min}
\begin{cases}
\Delta {}^{\mathcal D}\! E = 1 & \text{ in int\,} N \\
{}^{\mathcal D}\! E = 0 & \text{ on } \partial N.
\end{cases}
\end{equation}
\end{defn}
\begin{remark}\label{remark_L1}
Note that $ {}^{\mathcal D}\!  E $ can be obtained via exhaustion of $ N $. Indeed, let ${}^{\mathcal D}\! E_r $ denote the solution of
\begin{equation}\label{min_ex}
\begin{cases}
\Delta {}^{\mathcal D}\! E_r = 1 & \text{ in int\,} \Omega_r \\
{}^{\mathcal D}\! E_r = 0 & \text{ on } \partial \Omega_r
\end{cases}
\end{equation}
for an exhaustion $ \{ \Omega_r \} $ of $ N $. The same argument used in the previous section shows that $ {}^{\mathcal D}\! {E_r} $ either diverges at every point, or it converges to a function ${}^{\mathcal D}\! E $ which is the minimal solution of \eqref{min}.
Moreover,  applying the maximum principle  to the function $E- { }^{\mathcal{D}}\!E_r $ on  $ \Omega_r $ shows that $ E \geq { }^{\mathcal{D}}\!E_r $, so that $ E \geq {}^{\mathcal D}\! E $.
\end{remark}

We can now give a characterization of the $\mathcal{D}$-$L^1$-Liouville property similar to the one given above for the $ L^1$-Liouville property which relates the former to the nonintegrability of the Dirichlet Green function.

\begin{thm}
The following are equivalent:
\begin{itemize}
\item[1)] $ N $ is not $ \mathcal{D}$-$L^1$-Liouville;
\item[2)] ${ }^{\mathcal{D}}\!g(x, \cdot) $ is in $ L^1(N) $ for every $ x \in N $;
\item[3)] $ { }^{\mathcal{D}}\!E $ is finite for every $ x \in N $ and
\begin{equation*}
{ }^{\mathcal{D}}\!E (x) = \sum_{y \in N} { }^{\mathcal{D}}\!g(x,y) m(y).
\end{equation*}
\end{itemize}
\end{thm}

\begin{proof}
$ 1) \Leftrightarrow 2) $: Let $0\leq u \in L^1$ be a function violating the $ \mathcal{D}$-$L^1$-Liouville property. Then
\begin{equation*}
\begin{cases}
\Delta u \geq 0 & \text{ in } \mathrm{int}\, N \\
u \geq 0 & \text{ on } \partial N.
\end{cases}
\end{equation*}
Since there exists $ x_0 \in N $ such that $ u(x_0) > 0 $ we get that $ u > 0 $ on $ N $ by the minimum principle.

Let $ C > 1 $ be a large enough constant such that
\begin{equation*}
{}^{\mathcal D}\! g(x_0,x_0) \leq C u(x_0).
\end{equation*}
Applying now the maximum principle to the functions $ u- {}^{\mathcal D}\! g_r $ on an exhaustion $ \{ \Omega_r \} $ of $ N \setminus \{ x_0 \} $, where $ {}^{\mathcal D}\! g_r $ is harmonic, we get that ${}^{\mathcal D}\! g(x_0,y) \leq C u(y) $ holds for every $ y \in N $. It follows that
\begin{equation*}
\sum_{y \in N} {}^{\mathcal D}\! g(x_0,y) m(y) \leq C \sum_{y \in N} u(y) m(y) < \infty.
\end{equation*}
On the other hand, if $ {}^{\mathcal D}\! g(x, \cdot) \in L^1(N) $, then $ {}^{\mathcal D}\! g(x, \cdot) $ is a summable, non-negative, non-constant, super-harmonic function violating the $ \mathcal{D} $-$L^1$-Liouville.

$ 2) \Leftrightarrow 3) $ Consider an exhaustion $ \{ \Omega_r \} $ of $ N $ and let $ {}^{\mathcal D}\! g_r $ denote the corresponding Dirichlet Green function. It is clear by construction that the function $ \sum_{y \in \Omega_r}{}^{\mathcal D}\!  g_r(x,y) m(y) $ is the unique solution to problem \eqref{min_ex}, so that it coincides with ${}^{\mathcal D}\! E_r $. By passing to the limit as $ r \to \infty $ we get
\begin{equation*}
{}^{\mathcal D}\! E(x) = \lim_{r \to \infty}{}^{\mathcal D}\!  E_r(x) = \lim_{r \to \infty} \sum_{y \in N} {}^{\mathcal D}\!  g_r(x,y) \chi_{\Omega_r}(y) m(y) = \sum_{y \in N}{}^{\mathcal D}\!  g(x,y) m(y),
\end{equation*}
where we used the fact that ${}^{\mathcal D}\!  g_r(x,y) \nearrow {}^{\mathcal D}\! g(x,y) $ as $ r \to \infty $ and monotone convergence.
\end{proof}

As an immediate corollary we have the following result, which establishes a connection between the global $ L^1 $-Liouville and the $ \mathcal{D} $-$L^1$-Liouville of an (infinite) subgraph. This is an analogue of \cite[Corollary 30]{BPS} valid on Riemannian manifolds.

\begin{cor}
Let $ G $ be a graph and let $ N \subset G $ be an (infinite) subgraph of $ G $. If $ N $ is $ \mathcal{D}$-$L^1$-Liouville, then $ G $ is $ L^1 $-Liouville. Moreover, if there exists a sequence of subgraphs $ \{ N_k \}_k $ such that $ {}^{\mathcal D}\! E_{N_k}(x_0) \to \infty $ as $ k \to \infty $ for some $ x_0 \in D $, then $ G $ is $ L^1$-Liouville.
\end{cor}

\begin{proof}
Both statements follow by combining the previous theorem and Remark \ref{remark_L1}.
\end{proof}

We can now state and prove the main result of this section, which says that the validity of the $ L^1 $-Liouville property depends on the validity of the $ \mathcal{D}$-$L^1$-Liouville property on (at least) one of its ends.

\begin{thm}\label{end_thm}
Let $ G=(V,b,m) $ be a graph.
\begin{itemize}
\item[1)] Let $ N=(W,b_{|W\times W},m_{|W}) $ be a subgraph of $ G $ such that $ V \setminus W $ is finite.
    Then $ G $ is $ L^1$-Liouville if and only if $ N $ is $ \mathcal{D}$-$L^1$-Liouville.
\item[2)] $ G $ is $ L^1 $-Liouville if and only if at least one of its ends is $ \mathcal{D}$-$L^1$-Liouville.
\end{itemize}
\end{thm}

\begin{proof}
$1)$ In view of the previous corollary, we only need to prove that if $G$ is $L^1$-Liouville, then $ N $ is $ \mathcal{D}$-$L^1$-Liouville.

Let $ g $ be the Green kernel of $ G $. By assumption and Theorem~\ref{non-int}
\begin{equation*}
\sum_{x \in V} g(x,y) m(x) = \infty
\end{equation*}
for every $ y \in V $.  By construction  $ g $ is obtained as monotone limit of the Dirichlet Green kernels $g_r$  of an exhaustion $ \{ \Omega_r\}$ of $ G $, and we may assume that
$ V \setminus W \subset \mathrm{int\,}\Omega_1 $. Let $ u $ be a non-negative, non-constant  function such that $ \Delta u \geq 0 $ in $ \mathrm{int} N $ and $ u=0 $ in $ \partial N $. We need to show that $ u $ is not integrable. By the minimum principle $ u > 0 $ in $\mathrm{int \,}N$ and therefore  there exists a constant $ c_1 > 0 $ such that
\begin{equation*}
\min_{\partial \Omega_1} u = c_1>0.
\end{equation*}
On the other hand, having fixed  $ y_0 \in \Omega_1 $, we have that
\begin{equation*}
\max_{x \in \partial \Omega_1} g(x,y_0) = c_2 < \infty,
\end{equation*}
so that there exists a positive constant $ \lambda > 0 $ such that $ u(x) \geq \lambda g(x,y_0) $ for every $ x \in \partial \Omega_1 $. Since $ g \geq g_{r} $ it follows that $ u \geq \lambda  g_{r} $ on $ \partial \Omega_r \cup \partial \Omega_1$ and, by the maximum principle, we conclude that $ u \geq \lambda g_{r} $ in $ \Omega_r \setminus \Omega_1 $ and therefore $ u \geq \lambda g $ in $ V \setminus \Omega_1 $. Finally,
\begin{equation*}
\sum_{x \in W} u(x) m(x) \geq \lambda \sum_{x \in W \setminus \Omega_1} g(x,y_0) m(x) = \infty
\end{equation*}
as required to show that $ N $ is $ \mathcal{D}$-$L^1$-Liouville.

$ 2) $ We already know that if at least one end is $ \mathcal{D}$-$L^1$-Liouville, then $ G $ is $ L^1 $-Liouville. Next, suppose  by contradiction that no end of $ G $ is $ \mathcal{D}$-$L^1$-Liouville. Then there exists a finite set $ K $ such that $ G \setminus K = \bigcup_{i=1}^{m} F_i $, where $ F_i $ are ends of $ G $ and $ {}^{\mathcal D}\! {E_{F_i}} $ are finite functions for every $ i=1,...,m $. Since $ \partial F_i \cap \partial F_j = \emptyset $ if $ i \neq j $, setting $ {}^{\mathcal D}\! E_{F_i} = 0 $ outside $ F_i $ it follows that  ${}^{\mathcal D}\! E_{G \setminus K} = {}^{\mathcal D}\! E_{F_1} + \ldots + {}^{\mathcal D}\! E_{F_m} $, which is finite, and the contradiction follows from part $ i) $.
\end{proof}

%

\bigskip
\bigskip
\noindent




\begin{thebibliography}{99}

\bibitem{A} A. Adriani, \textit{A note on comparison theorems for graphs}, J. Math. Anal. Appl. \textbf{503} (2021).

\bibitem{AS} A. Adriani and A.G. Setti, \textit{Inner-Outer Curvatures, Ricci-Ollivier Curvature and Volume Growth  of Graphs}, Proc. Amer. Math. Soc. \textbf{149} (2021), no. 11, 4609--4621.

\bibitem{BB09} C. B\"ar and G.P. Bessa, \textit{Stochastic completeness and volume growth}, Proc. Amer. Math. Soc. \textbf{138} (2010), no. 7, 2629--2640.

\bibitem{BPS} G.P. Bessa, S. Pigola and A.G. Setti, \textit{On the $ L^1 $-Liouville property of stochastically incomplete manifolds}, Potential Anal. \textbf{39} (2013), 313--324.


\bibitem{CLMP} D. Cushing, S. Liu, F. M\"unch and N. Peyerimhoff, \textit{Curvature calculations for antitrees}, Analysis and geometry on graphs and manifolds, 21-54, London Mathematical Society Lecture Note Series, \textbf{461}. Cambridge University Press, Cambridge, 2020.

\bibitem{D84} J. Dodziuk, \textit{Difference equations, isoperimetric inequality and transience of certain random walks}, Trans. Amer. Math. Soc. \textbf{284} (1984), no. 2, 787--794.

\bibitem{DK88} J. Dodziuk and L. Karp, \textit{Spectral and function theory for combinatorial Laplacians},  Geometry of random motion (Ithaca, N.Y., 1987), 25--40, Contemp. Math., \textbf{73}, Amer. Math. Soc., Providence, RI, 1988.

\bibitem{F14}  M. Folz, \textit{Volume growth and stochastic completeness of graphs}, Trans. Amer. Math. Soc. \textbf{366} (2014), no. 4, 2089--2119.

\bibitem{Gri} A. Grigor'yan, \textit{Stochastically complete manifolds and summable harmonic functions}, Izv. Akad.
Nauk SSSR Ser. Mat. \textbf{52} (1988), 1102–1108; translation in Math. USSR-Izv. \textbf{33} (1989),  425--432.

\bibitem{Gri18} A. Grigor'yan, Introduction to analysis on graphs, University Lecture Series, \textbf{71}, Amer. Math. Soc., Providence, RI, 2018.

\bibitem{HJ14} B. Hua, J. Jost, \textit{$L^q$ harmonic functions on graphs}, Isr. J. Math. \textbf{202} (2014),  475--490.

\bibitem{HK} B. Hua and M. Keller, \textit{Harmonic functions of general graph Laplacians}, Calc. Var. Partial Differential Equations \textbf{51} (2014), 343--362.


\bibitem{XH11} X. Huang, \textit{Stochastic incompleteness for graphs and the weak Omori-Yau maximum principle}. J. Math. Anal. Appl. \textbf{379} (2011), no. 2,  764--782.

\bibitem{XH14} X. Huang, \textit{A note on the volume growth criterion for stochastic completeness of weighted graphs}, Potential Analysis \textbf{40} (2014), 117--142.


\bibitem{HKS20} X. Huang, M. Keller and M. Schmidt, \textit{On the uniqueness class, stochastic completeness and volume growth for graphs}, Trans. Amer. Math. Soc.  \textbf{373} (2020),  no. 12, 8861--8884.

\bibitem{Karp82} L. Karp, \textit{Subharmonic functions on real and complex manifolds}, Math. Z. \textbf{179} (1982), no. 4, 535--554.

\bibitem{K15} M. Keller, \textit{Intrinsic metrics on graphs: A survey}, Mathematical Technology of Networks, 81--119, Springer Proc. Math. Stat.,  \textbf{128} Springer, Cham, 2015.

\bibitem{KL12} M. Keller, D. Lenz, \textit{Dirichlet forms and stochastic completeness of graphs and subgraphs}, J. Reine Angew. Math. \textbf{666} (2012),  189–223.

\bibitem{KLW} M. Keller, D. Lenz and R.K. Wojciechowski, \textit{Volume growth, spectrum and stochastic completeness of infinite graphs}. Math. Z. \textbf{274} (2013), no. 3-4, 905--932.

\bibitem{KLW21} M. Keller, D. Lenz and R.K. Wojciechowski, Graphs and discrete Dirichlet spaces, Grundlehren der mathematischen Wissenschaften, \textbf{358}, Springer, 2021.

\bibitem{LLY11} Y. Lin, L. Lu and S.-T. Yau, \textit{Ricci curvature of graphs}, Tohoku Math. J.  \textbf{63} (2011), no. 4, 605--627.

\bibitem{M09} J. Masamune, \textit{A Liouville property and its application to the Laplacian of an infinite graph,} Spectral analysis in geometry and number theory, 103--115, Contemp. Math., \textbf{484}, Amer. Math. Soc., Providence, RI, 2009.

\bibitem{MW19} F. M\"unch and R.K. Wojciechowski, \textit{Ollivier Ricci curvature for general graph Laplacians: heat equation, Laplacian comparison, non-explosion and diameter bounds}, Adv. Math. \textbf{356} (2019).

\bibitem{Oll07}  Y. Ollivier, \textit{Ricci curvature of metric spaces}, C. R. Math. Sci. Paris \textbf{345} (2007), no. 11, 643--646.

\bibitem{PPS} L.F. Pessoa, S. Pigola and A.G. Setti, \textit{Dirichlet parabolicity and $L^1$-Liouville property under localized geometric conditions}, J. Funct. Anal. \textbf{273} (2017), no. 2, 652--693.



\bibitem{Sturm94} K.-T. Sturm, \textit{Analysis on local Dirichlet spaces I. Recurrence, conservativeness and $L^p$-Liouville properties}, J. Reine Angew. Math. \textbf{456} (1994), 173--196.

\bibitem{Web10} A. Weber, \textit{Analysis of the physical Laplacian and the heat flow on a locally finite graph}, J. Math. Anal. Appl. \textbf{370} (2010), no. 1, 146--158.

\bibitem{RKW1} R.K. Wojciechowski, \textit{Stochastic completeness of graphs}, ProQuest LLC, Ann Arbor, MI (2008). Thesis (Ph.D.)--City University of New York.

\bibitem{RKW11} R.K. Wojciechowski, \textit{Stochastically incomplete manifolds and graphs}, Random Walks, Boundaries and Spectra, Progress in Probability \textbf{64} Springer, Basel, 2011.


\bibitem{RKW2}  R.K. Wojciechowski, \textit{The Feller property for graphs}, Trans. Amer. Math. Soc. \textbf{369} (2017), no. 6, 4415--4431.

\bibitem{RKW21} R.K. Wojciechowski, \textit{Stochastic Completeness of Graphs: Bounded Laplacians, Intrinsic Metrics, Volume Growth and Curvature}, J. Fourier Anal. Appl. \textbf{27} (2021), no. 2.

\bibitem{Yau} S.-T. Yau, \textit{Some function-theoretic properties of complete Riemannian manifold and their applications to geometry}, Indiana University Math. J.
\textbf{25} (1976), no. 7, 659--670.

\end{thebibliography}
\end{document}